\documentclass{amsart}

\usepackage{amssymb}
\usepackage{amsmath}
\usepackage{amsthm}
\usepackage{amsfonts}

\usepackage{latexsym}

\usepackage[dvips]{epsfig}
\usepackage{amsbsy}
\usepackage{amsgen}
\usepackage{amscd}
\usepackage{amsopn}
\usepackage{amstext}
\usepackage{amsxtra}
\usepackage{xypic}

\newtheorem{theorem}{Theorem}[section]
\newtheorem{definition}[theorem]{Definition}
\newtheorem{example}[theorem]{Example}
\newtheorem{remark}[theorem]{Remark}
\newtheorem{lem}[theorem]{Lemma}
\newtheorem{proposition}[theorem]{Proposition}
\newtheorem{corollary}[theorem]{Corollary}

\newtheorem{sublem}[theorem]{Sublemma}
\newtheorem{some facts}[theorem]{Some facts about 
${\bf M}^{\mathcal{C}^{op}}_{\ast}$}

\begin{document}
\title[Complements on enriched precategories]
{Complements on enriched precategories}
\author[{A. E.} {Stanculescu}]{{Alexandru E.} {Stanculescu}}
\address{\newline Department of Mathematics and Statistics,
\newline Masaryk University,  Kotl\'{a}{\v{r}}sk{\'{a}} 2,\newline
611 37 Brno, Czech Republic}
\email{stanculescu@math.muni.cz}

\begin{abstract}
We make some remarks on the foundations of the homotopy
theory of enriched precategories, as exposed in Carlos 
Simpson's book ``Homotopy theory of higher categories''.
\end{abstract}

\maketitle

The main goal of this note is to present an alternative point of view on 
some results from Carlos Simpson's book ``Homotopy theory of higher 
categories'' \cite{Si}. For a part of these results, the alternative approach
is already hinted at by Simpson. Save section 3 below, we target 
essentially the content of chapters 9-11, sections 12.1-12.3, 12.6 
and some facts scattered through chapters 13 and 14. There is one 
major exception: we do not treat here the so-called injective model structure
on the category of enriched precategories with fixed set of objects.

In section 1 we introduce one of the two categories of interest to us and
we highlight some of its properties. In section 2 we first address the Reedy 
model structure on enriched precategories with fixed set
of objects and study the behaviour of this model structure 
under change of diagram and base category. Next we address
the projective model structure on enriched precategories with fixed set
of objects and compare it with the Reedy model structure.
Section 3 reviews Lurie's proof \cite{Lu} of the Quillen equivalence
between the projective model structure on enriched precategories
and that of enriched categories, in the (critical)  fixed set of objects
case. We observe that his result holds under weaker assumptions
on the base category. In section 4 we recall the construction of the
category of enriched precategories--the other category of interest to us,
and we remark that it is a bifibration over the category of sets. 
Section 5 introduces the fibred Reedy model structure on 
the category of enriched precategories and singles out
a certain weak factorization system on it.
In section 6 we introduce the fibred projective model 
structure on the category of enriched precategories.
In the appendix we recall a couple of results concerning
left Bousfield localization.
\\

{\bf Notation.} The terminal object of a category, when it exists, 
is denoted by $\ast$.
\\

Independent from the results of this note, we make
below a list of some facts from \cite{Si} we wish to 
understand better in the future.
\\

(1)  On page 291, why is the sentence ``Furthermore,...(see Lemma 10.3.2)''
true?

(2) One may compare (the proof of) 13.7.3 with the paragraph on page 447
``A map...Thus, $p$ is a global weak equivalence.''

(3) 14.3.5 is used in the proof of 19.2.1(PGM6). It is not mentioned
in the statement of 19.2.1 that the class of weak equivalences of 
$\mathcal{M}$ is closed under transfinite compositions.

(4) In the proof of 18.7.1, on page 438, the maps 
$f|_{\mathcal{B}\times \{v_{1}\}}$ and $q$ do not seem
to have the same target.

(5) In the proof of 19.3.1, why is the left vertical map in the diagram a cofibration?

\section{The category ${\bf M}^{\mathcal{C}^{op}}_{\ast}$}

Let $\mathcal{C}$ be a small Reedy category and 
let $F^{n}\mathcal{C}$ be the $n$-filtration of $\mathcal{C}$
\cite[15.1.22]{Hi}; then $F^{0}\mathcal{C}$ is a discrete category \cite[15.1.23]{Hi}.
We denote the inclusion $F^{0}\mathcal{C}\subset \mathcal{C}$ by $\sigma_{0}$. 
Let {\bf M} be a category. The restriction functor $\sigma_{0}^{\ast}:{\bf M}^{\mathcal{C}^{op}}
\rightarrow {\bf M}^{(F^{0}\mathcal{C})^{op}}$ has a left adjoint $\sigma_{0{!}}$
provided that {\bf M} has coproducts.
\begin{definition}
Let {\bf M} be a category with terminal object. We denote by 
${\bf M}^{\mathcal{C}^{op}}_{\ast}$ the full subcategory of 
${\bf M}^{\mathcal{C}^{op}}$  on objects {\bf X} such that 
$\sigma_{0}^{\ast}{\bf X}=\ast$. We let $K$ be the inclusion functor 
${\bf M}^{\mathcal{C}^{op}}_{\ast}\subset {\bf M}^{\mathcal{C}^{op}}$.
\end{definition}
\begin{some facts}
{\rm
$(a)$ $K$ creates colimits indexed by connected digrams and it has a left 
adjoint $r$ calculated as follows. 
If ${\bf X} \in {\bf M}^{\mathcal{C}^{op}}$, 
then one has a pushout square
\[
   \xymatrix{
\sigma_{0!}\sigma_{0}^{\ast}{\bf X} \ar[r] \ar[d] & {\bf X} \ar[d]\\
\sigma_{0!}\sigma_{0}^{\ast}\ast \ar[r] & r{\bf X}\\
}
  \]

$(b)$ ${\bf M}^{\mathcal{C}^{op}}_{\ast}$ is an accessible
category if {\bf M} is. 

$(c)$ Suppose that {\bf M} is a closed category with monoidal product $\otimes$. 
Write $Y^{X}$ for the internal hom of two objects $X$, $Y$ of {\bf M}. 
Then ${\bf M}^{\mathcal{C}^{op}}$ is tensored, cotensored and enriched over 
{\bf M}, with tensor, cotensor and {\bf M}-hom defined as
$$(A\bigstar {\bf X})(c)=A\otimes {\bf X}(c)$$ 
$$({\bf X}^{A})(c)={\bf X}(c)^{A}$$
$$\underline{Map}({\bf X},{\bf Y})=\underset{c} \int {\bf Y}(c)^{{\bf X}(c)}$$
It follows that ${\bf M}^{\mathcal{C}^{op}}_{\ast}$ is tensored, 
cotensored and enriched over {\bf M}, with tensor, cotensor 
and {\bf M}-hom defined by the formulae
$A\bigstar {\bf X}=r(A\bigstar K{\bf X})$, 
${\bf X}^{A}=(K{\bf X})^{A}$ and 
$\underline{Map}({\bf X},{\bf Y})=
\underline{Map}(K{\bf X},K{\bf Y})$.
The adjunction $(r,K)$ becomes an {\bf M}-adjunction.

$(d)$ For every small category $I$ there is an isomorphism of 
categories $$({\bf M}^{I})^{\mathcal{C}^{op}}_{\ast}\cong 
({\bf M}^{\mathcal{C}^{op}}_{\ast})^{I}$$
and these categories are isomorphic in turn to 
the full subcategory of ${\bf M}^{\mathcal{C}^{op}\times I}$
on objects {\bf X} having the property that ${\bf X}(c_{0},i)=\ast$,
for $c_{0}\in F^{0}\mathcal{C}$ and $i\in I$.

$(e)$ (Change of diagram) Let $F:\mathcal{C}\rightarrow \mathcal{D}$ be a functor
between Reedy categories which preserves the $0$-filtration. 
The induced functor $F^{\ast}:{\bf M}^{\mathcal{C}^{op}}_{\ast}
\rightarrow {\bf M}^{\mathcal{D}^{op}}_{\ast}$ has a left adjoint $F_{!}'$
 constructed in such a way that it makes the diagram
\[
   \xymatrix{
{\bf M}^{\mathcal{C}^{op}} \ar @<2pt> [r]^{F_{!}} \ar @<2pt> [d]^{r} &
{\bf M}^{\mathcal{D}^{op}} \ar @<2pt> [l]^{F^{\ast}} \ar @<2pt> [d]^{r}\\
{\bf M}^{\mathcal{C}^{op}}_{\ast}  \ar @<2pt> @{..>} [r]^{F_{!}'} \ar @<2pt> [u]^{K} 
& {\bf M}^{\mathcal{D}^{op}}_{\ast} \ar @<2pt> [l]^{F^{\ast}} \ar @<2pt> [u]^{K}\\
  }
  \]
commutative in the obvious sense. Since $K$ is full and faithful, 
one has $F_{!}'=rF_{!}K$.

$(f)$ Let $F:{\bf M}_{1}\rightarrow {\bf M}_{2}$ be a functor which 
preserves the terminal object. Then $F$ induces a functor 
$F:{\bf M}^{\mathcal{C}^{op}}_{1\ \ast}
\rightarrow {\bf M}^{\mathcal{C}^{op}}_{2\ \ast}$.
If $F:{\bf M}_{1}\rightarrow {\bf M}_{2}$ is full and faithful, then so is $F$. 

$(g)$ (Change of base category) Let $F:{\bf M}_{1}\rightleftarrows {\bf M}_{2}:G$ 
be an adjoint pair. The induced functor $G:{\bf M}^{\mathcal{C}^{op}}_{2\ \ast}
\rightarrow {\bf M}^{\mathcal{C}^{op}}_{1\ \ast}$ has a left adjoint $F'$
 constructed in such a way that it makes the diagram
\[
   \xymatrix{
{\bf M}_{1}^{\mathcal{C}^{op}} \ar @<2pt> [r]^{F} 
\ar @<2pt> [d]^{r_{1}} & {\bf M}_{2}^{\mathcal{C}^{op}} 
\ar @<2pt> [l]^{G} \ar @<2pt> [d]^{r_{2}}\\
{\bf M}^{\mathcal{C}^{op}}_{1\ \ast}  \ar @<2pt> @{..>} [r]^{F'} 
\ar @<2pt> [u]^{K_{1}} & {\bf M}^{\mathcal{C}^{op}}_{2\ \ast} 
\ar @<2pt> [l]^{G} \ar @<2pt> [u]^{K_{2}}\\
  }
  \]
commutative in the obvious sense. Since $K_{1}$ is full and faithful, one has 
$F'=r_{2}FK_{1}$. If $F\ast\cong \ast$, then $F'\ast\cong \ast$.}
\end{some facts}

\section{The Reedy and projective model structures for Segal {\bf M}-categories
with fixed set of objects}

\subsection{The Reedy model structure}
Recall from definition 1.1 the functor $K$.
\begin{proposition}{\rm (compare with \cite[11.7 and 12.3.1]{Si})}
Let  {\bf M} be a model category and $\mathcal{C}$ a 
small Reedy category. Let ${\bf M}^{\mathcal{C}^{op}}$
have the Reedy model structure. Then the category 
${\bf M}^{\mathcal{C}^{op}}_{\ast}$ admits a model 
category structure with the classes of 
weak equivalences, cofibrations and fibrations defined 
via the functor $K$. ${\bf M}^{\mathcal{C}^{op}}_{\ast}$
is cofibrantly generated if ${\bf M}^{\mathcal{C}^{op}}$ is. 
${\bf M}^{\mathcal{C}^{op}}_{\ast}$ is left proper if {\bf M} is.
\end{proposition}
\begin{proof}
The lifting axiom of a model category is clear. 
The factorization axiom is shown inductively on the 
degree of the objects of $\mathcal{C}$, exactly as in \cite[15.3.16]{Hi}.
The only difference with \emph{loc. cit.} is in degree zero, when
we choose ${\bf Z}_{\alpha}=\ast$, for every object $\alpha$
of $\mathcal{C}$ of degree zero. 
Suppose now that ${\bf M}^{\mathcal{C}^{op}}$ is 
cofibrantly generated. Let $\mathbb{I}$ and $\mathbb{J}$
be generating sets of cofibrations and trivial cofibrations, 
respectively. By 1.2$(a)$ the sets $r(\mathbb{I})$ and 
$r(\mathbb{J})$ permit the small object argument.
They can be chosen to be generating sets of cofibrations 
and trivial cofibrations for the model category
${\bf M}^{\mathcal{C}^{op}}_{\ast}$. 
Left properness is straightforward.
\end{proof}
\begin{proposition}
Let $F:\mathcal{C}\rightarrow \mathcal{D}$ be a functor 
between Reedy categories which preserves the $0$-filtration.
Let {\bf M} be a model category. Suppose that the induced adjoint 
pair $F_{!}:{\bf M}^{\mathcal{C}^{op}}\rightleftarrows 
{\bf M}^{\mathcal{D}^{op}}:F^{\ast}$ is a Quillen pair.
Then the induced adjoint pair 
$F_{!}':{\bf M}^{\mathcal{C}^{op}}_{\ast}
\rightleftarrows {\bf M}^{\mathcal{D}^{op}}_{\ast}:F^{\ast}$ 
 $(1.2(e))$ is a Quillen pair.
\end{proposition}
We recall from \cite[Definition 3.16.1]{Ba} that a functor between 
Reedy categories is a {\bf morphism} if it preserves the inverse 
and direct subcategories.

Let us fix a small Reedy category $\mathcal{C}$ and a 
degree preserving morphism $p:\mathcal{C}\rightarrow \Delta$
which is a fibration. We denote the fibre category of $p$ over 
$[n]\in \Delta$ by $\mathcal{C}_{n}$, 
and the natural functor $\mathcal{C}_{n}
\rightarrow \mathcal{C}$ by $\sigma_{n}$. One has 
$\mathcal{C}_{0}=F^{0}\mathcal{C}$. 

For each $n\geq 1$, let $\alpha^{k}:[1]\rightarrow [n]$
be the map in $\Delta$ defined as $\alpha^{k}(0)=k$ and 
$\alpha^{k}(1)=k+1$, where $0\leq k<n$. If $n\geq 2$, one has 
$\alpha^{k}d^{0}=\alpha^{k+1}d^{1}$, for $0\leq k\leq n-2$.
For each $c\in \mathcal{C}$ of degree $n\geq 1$,
let $(\alpha^{k})^{\ast}c\rightarrow c$ be a
cartesian lifting of $\alpha^{k}$, $0\leq k<n$, and 
$(d^{i})^{\ast}(\alpha^{k})^{\ast}c\rightarrow (\alpha^{k})^{\ast}c$
be a cartesian lifting of $d^{i}:[0]\rightarrow [1]$, $0\leq i\leq 1$.
We obtain a commutative diagram in $\mathcal{C}$
\[
   \xymatrix{
&  (d^{0})^{\ast}(\alpha^{k})^{\ast}c=
(d^{1})^{\ast}(\alpha^{k+1})^{\ast}c \ar[dl] \ar[dr]\\
(\alpha^{k})^{\ast}c \ar[dr] & & (\alpha^{k+1})^{\ast}c \ar[dl]\\
& c\\
  }
  \]
Let now {\bf M} be a left proper, combinatorial model category.
For $c\in \mathcal{C}$ we denote by $ev_{c}$ the evaluation at $c$
functor ${\bf M}^{\mathcal{C}^{op}}\rightarrow {\bf M}$. $ev_{c}$
has a left adjoint $F_{c}$ which sends $A\in {\bf M}$ to
$$F_{c}^{A}(c')=\underset{\mathcal{C}(c',c)}\coprod A$$
Let $W$ be the set considered in \cite[Proposition A.5]{Du}.
We denote by $\mathfrak{S}$ the set 
$$\{rF_{(\alpha^{0})^{\ast}c}^{A}\underset{rF_{(d^{0})^{\ast}(\alpha^{0})^{\ast}c}^{A}}
\cup...\underset{rF_{(d^{0})^{\ast}(\alpha^{n-2})^{\ast}c}^{A}}
\cup rF_{(\alpha^{n-1})^{\ast}c}^{A}
\rightarrow rF_{c}^{A}\}_{\{c, deg(c)\geq 1\}\times \{A\in W\}}$$
where the map is induced by commutative diagrams as above.
\begin{theorem} 
Let {\bf M} be a left proper, combinatorial model category. The category
${\bf M}^{\mathcal{C}^{op}}_{\ast}$ admits a left proper, 
combinatorial model category structure with 
the cofibrations of ${\bf M}^{\mathcal{C}^{op}}_{\ast}$ as cofibrations 
and the $\mathfrak{S}$-local equivalences as weak equivalences. We 
denote this model structure by ${\bf M}^{\mathcal{C}^{op}}_{\ast,S}$, where
``S'' stands for Segal. If {\bf X} is an object of ${\bf M}^{\mathcal{C}^{op}}_{\ast}$,
then {\bf X} is fibrant in ${\bf M}^{\mathcal{C}^{op}}_{\ast,S}$ 
if and only if {\bf X} is fibrant in ${\bf M}^{\mathcal{C}^{op}}_{\ast}$ 
and for every object $c$ of $\mathcal{C}$ with $deg(c)\geq 1$, the map
$${\bf X}(c)\rightarrow {\bf X}((\alpha^{0})^{\ast}c)
\times...\times {\bf X}((\alpha^{deg(c)-1})^{\ast}c)$$
is a weak equivalence of {\bf M}.
\end{theorem}
\begin{proof}
The model structure exists by Smith's theorem \cite{Du} 
applied to the model category ${\bf M}^{\mathcal{C}^{op}}_{\ast}$ 
from proposition 2.1 and the set $\mathfrak{S}$. It remains
to show that the $\mathfrak{S}$-local objects are the ones 
mentioned in the statement of the theorem. The proof is the 
same as the proof of \cite[Theorem 5.2(c)]{Du}.
To begin with, notice that to give a map 
$$rF_{(\alpha^{0})^{\ast}c}^{A}
\underset{rF_{(d^{0})^{\ast}(\alpha^{0})^{\ast}c}^{A}}
\cup...\underset{rF_{(d^{0})^{\ast}(\alpha^{deg(c)-2})^{\ast}c}^{A}}
\cup rF_{(\alpha^{deg(c)-1})^{\ast}c}^{A} \rightarrow {\bf X}$$
where $deg(c)\geq 1$ and $A\in W$, is to give a map $A \rightarrow 
{\bf X}((\alpha^{0})^{\ast}c)
\times...\times {\bf X}((\alpha^{deg(c)-1})^{\ast}c)$.
It follows that $$rF_{(\alpha^{0})^{\ast}c}^{A}
\underset{rF_{(d^{0})^{\ast}(\alpha^{0})^{\ast}c}^{A}}
\cup...\underset{rF_{(d^{0})^{\ast}(\alpha^{deg(c)-2})^{\ast}c}^{A}}
\cup rF_{(\alpha^{deg(c)-1})^{\ast}c}^{A}$$ is cofibrant in 
${\bf M}^{\mathcal{C}^{op}}_{\ast}$. Using
\cite[17.4.16(2)]{Hi} and \cite[Proposition A.5]{Du}
we arrive at the desired characterization of 
$\mathfrak{S}$-local objects.
\end{proof}
\begin{example}
{\rm Let {\bf Cat} be the category of all small 
categories and {\bf S} the category of simplicial sets.
For a small category $\mathcal{C}$, we denote by 
$y_{\mathcal{C}}:\mathcal{C}\rightarrow Set^{\mathcal{C}^{op}}$,
or simply $y$, the Yoneda embedding.

Let $\mathcal{C}$ be a small category and 
$\Psi:\mathcal{C}^{op}\rightarrow {\bf Cat}$ a functor.
The Grothendieck construction $\underset{\mathcal{C}}\int (\Psi)$
is the category with objects $(c,x)$ where $c\in Ob\mathcal{C}$
and $x\in Ob\Psi(c)$, and arrows  $(c,x)\rightarrow (d,y)$ are pairs
$(u,f)$ with $u:c\rightarrow d$ in $\mathcal{C}$ and 
$f:x\rightarrow \Psi(u)(y)$ in $\Psi(c)$. The projection
$\underset{\mathcal{C}}\int (\Psi)\rightarrow \mathcal{C}$ 
is a fibration.

Let $\mathcal{C}$ be a small Reedy category and
let $X\in Set^{\mathcal{C}^{op}}$. The comma category 
$(y\downarrow X)$ is the Grothendieck construction
of the composite functor $\mathcal{C}^{op}\overset{X}\rightarrow 
Set\overset{D}\rightarrow {\bf Cat}$, where $D:Set\rightarrow {\bf Cat}$
is the discrete category functor. $(y\downarrow X)$ becomes a 
Reedy category and the projection $(y\downarrow X)\rightarrow \mathcal{C}$
a degree preserving morphism of Reedy categories.

Let $N:{\bf Cat}\rightarrow {\bf S}$ be the nerve functor.
If $\mathcal{C}$ is a small category, we put $\Delta \mathcal{C}=
(y\downarrow N(\mathcal{C}))$ and $\Delta^{op}\mathcal{C}=
(\Delta\mathcal{C})^{op}$. If 
$([n],c_{0}\rightarrow...\rightarrow c_{n})$
is an object of $\Delta \mathcal{C}$ with $n\geq 2$, the cartesian lifting
of $\alpha^{k}:[1]\rightarrow [n]$ is $([1],c_{k}\rightarrow c_{k+1})
\rightarrow ([n],c_{0}\rightarrow...\rightarrow c_{n})$.

Let $\iota :Set\rightarrow {\bf Cat}$ be 
the indiscrete/chaotic category functor, right 
adjoint to the set of objects functor.
If $S$ is a set, one has $N(\iota S)_{n}=
\underset{[0]\rightarrow [n]}\prod S$. 
We put $\Delta S=\Delta\iota S$. If $([n],s_{0},...,s_{n})$
is an object of $\Delta S$ with $n\geq 2$, the cartesian lifting
of $\alpha^{k}:[1]\rightarrow [n]$ is $([1],s_{k},s_{k+1})
\rightarrow ([n],s_{0},...,s_{n})$. In \cite[Section 5]{Be1}, J. Bergner 
has made an early use of the category $\Delta S$, in the same context as ours. 
The existence of the model structure ${\bf M}^{\Delta^{op}S}_{\ast,S}$ 
was proved in \cite[12.3.2]{Si}, using a different method. An object {\bf X} of 
${\bf M}^{\Delta^{op}S}_{\ast}$ is fibrant in 
${\bf M}^{\Delta^{op}S}_{\ast,S}$ if and only if {\bf X} 
is fibrant in ${\bf M}^{\Delta^{op}S}_{\ast}$ and for every object
$([n],s_{0},...,s_{n})$ of $\Delta S$ $(n\geq 1)$, the map
$${\bf X}(([n],s_{0},...,s_{n}))\rightarrow {\bf X}(([1],s_{0},s_{1}))
\times {\bf X}(([1],s_{1},s_{2}))\times...\times {\bf X}(([1],s_{n-1},s_{n}))$$
is a weak equivalence of {\bf M}.}
\end{example}

\begin{remark}
{\rm One can perform the localization at a different set of maps. 
Here is an example \cite[Section 6]{Be2}.
For each $n\geq 1$, let $\gamma^{k}:[1]\rightarrow [n]$
be the map in $\Delta$ defined as $\gamma^{k}(0)=0$ and 
$\gamma^{k}(1)=k+1$, where $0\leq k<n$. If $n\geq 2$, one has 
$\gamma^{k}d^{1}=\gamma^{k+1}d^{1}$, for $0\leq k\leq n-2$.
For each $c\in \mathcal{C}$ of degree $n\geq 1$,
let $(\gamma^{k})^{\ast}c\rightarrow c$ be a
cartesian lifting of $\gamma^{k}$, $0\leq k<n$, and 
$(d^{1})^{\ast}(\alpha^{k})^{\ast}c\rightarrow (\alpha^{k})^{\ast}c$
be a cartesian lifting of $d^{1}:[0]\rightarrow [1]$.
We obtain a commutative diagram in $\mathcal{C}$
\[
   \xymatrix{
&  (d^{1})^{\ast}(\gamma^{k})^{\ast}c=
(d^{1})^{\ast}(\gamma^{k+1})^{\ast}c \ar[dl] \ar[dr]\\
(\gamma^{k})^{\ast}c \ar[dr] & & (\gamma^{k+1})^{\ast}c \ar[dl]\\
& c\\
  }
  \]
and the set of maps}
$$\{rF_{(\gamma^{0})^{\ast}c}^{A}\underset{rF_{(d^{1})^{\ast}(\gamma^{0})^{\ast}c}^{A}}
\cup...\underset{rF_{(d^{1})^{\ast}(\gamma^{0})^{\ast}c}^{A}}
\cup rF_{(\gamma^{n-1})^{\ast}c}^{A}
\rightarrow rF_{c}^{A}\}_{\{c, deg(c)\geq 1\}\times \{A\in W\}}$$
\end{remark}
We shall study now the behaviour of the model category
${\bf M}^{\mathcal{C}^{op}}_{\ast,S}$ under change 
of diagram and base category.
\begin{proposition}
Let $p:\mathcal{C}\rightarrow \Delta$ and $q:\mathcal{D}\rightarrow \Delta$
be degree preserving morphisms of Reedy categories which are fibrations.
Let $F:\mathcal{C}\rightarrow \mathcal{D}$ be a degree preserving morphism 
of Reedy categories such that $F$ is a fibred functor. 
Let {\bf M} be a left proper, combinatorial model category.
Suppose that the induced adjoint pair
$F_{!}:{\bf M}^{\mathcal{C}^{op}}\rightleftarrows 
{\bf M}^{\mathcal{D}^{op}}:F^{\ast}$ is a Quillen pair.
Then the induced adjoint pair 
$F_{!}':{\bf M}^{\mathcal{C}^{op}}_{\ast,S}
\rightleftarrows {\bf M}^{\mathcal{D}^{op}}_{\ast,S}:F^{\ast}$ 
$(1.2(e))$ is a Quillen pair.
\end{proposition}
\begin{proof}
From proposition 2.2 it suffices \cite[8.5.4(3)]{Hi} to show 
that $F^{\ast}$ preserves fibrations between fibrant objects. 
But this is clear from the assumptions on $F$.
\end{proof}
We recall \cite[Definition 3.16.3]{Ba} that a morphism $f:\mathcal{C}
\rightarrow \mathcal{D}$ of Reedy categories is a {\bf right fibration}
if for every model category {\bf M}, the adjoint pair
$f_{!}:{\bf M}^{\mathcal{C}^{op}}\rightleftarrows 
{\bf M}^{\mathcal{D}^{op}}:f^{\ast}$ is a Quillen pair.
\begin{lem}
For every map $f:X\rightarrow Y$ of simplicial sets, 
$f:(y\downarrow X)\rightarrow (y\downarrow Y)$ is a right fibration.
\end{lem}
\begin{proof}
It suffices to prove that $\partial(\overrightarrow {(([n],y)\downarrow 
\overrightarrow{f})}\downarrow \lambda)$ is empty or connected,
where $\lambda:([n],y)\rightarrow ([m],f(x))$ and $\lambda$
is a monomorphism. A commutative diagram
\[
   \xymatrix{
& ([n],y) \ar[dr]^{\beta} \ar[dl]_{\alpha} \ar[dd]^{\lambda}\\
([p],f(a)) \ar[dr]_{\gamma} & & ([q],f(b)) \ar[dl]^{\delta}\\
& ([m],f(x))\\
}
  \]
in which $\alpha,\beta,\delta,\gamma$ are monomorphisms
can be completed to a commutative diagram
\[
   \xymatrix{
& ([n],y) \ar[dr]^{\beta} \ar[dl]_{\alpha} \ar[d]^{\eta}\\
([p],f(a)) \ar[dr]_{\gamma} & ([r],f(u^{\ast}x)) 
\ar[l]_{\theta} \ar[d]^{u} \ar[r]^{\varepsilon} & ([q],f(b)) \ar[dl]^{\delta}\\
& ([m],f(x))\\
}
  \]
in which $\theta,\varepsilon,\eta,u$ are monomorphisms
and $\lambda=u\eta$.
\end{proof}
\begin{corollary}
For any map $f:X\rightarrow Y$ of simplicial sets and any 
model category {\bf M}, the adjoint pair $$(y\downarrow f)_{!}:
{\bf M}^{(y\downarrow X)^{op}}\rightleftarrows
{\bf M}^{(y\downarrow Y)^{op}}:(y\downarrow f)^{\ast}$$ 
is a Quillen pair. In particular, if $f:S\rightarrow T$ is a function, 
the adjoint pair $$f_{!}':{\bf M}^{\Delta^{op}S}_{\ast}\rightleftarrows
{\bf M}^{\Delta^{op}T}_{\ast}:f^{\ast}$$ is a Quillen pair.
\end{corollary}
\begin{corollary}
Let {\bf M} be a left proper, combinatorial model category. If $f:S\rightarrow T$
is a function, then the adjoint pair $f_{!}':{\bf M}^{\Delta^{op}S}_{\ast,S}
\rightleftarrows {\bf M}^{\Delta^{op}T}_{\ast,S}:f^{\ast}$ 
is a Quillen pair.
\end{corollary}
\begin{proof}
Apply corollary 2.8 and proposition 2.6.
\end{proof}
\begin{proposition} {\rm (compare with \cite[12.6.2 and 14.7.3]{Si})}
Let $F:{\bf M}_{1}\rightleftarrows {\bf M}_{2}:G$ be a Quillen pair
between left proper, combintorial model categories. Then the induced
adjoint pairs $F':{\bf M}^{\mathcal{C}^{op}}_{1\ \ast}
\rightleftarrows {\bf M}^{\mathcal{C}^{op}}_{2\ \ast}:G$
and $F':{\bf M}^{\mathcal{C}^{op}}_{1\ \ast,S}
\rightleftarrows {\bf M}^{\mathcal{C}^{op}}_{2\ \ast,S}:G$ 
$(1.2(g))$ are Quillen pairs. If  
$F':{\bf M}^{\mathcal{C}^{op}}_{1\ \ast}
\rightarrow {\bf M}^{\mathcal{C}^{op}}_{2\ \ast}$
preserves weak equivalences then so does
$F':{\bf M}^{\mathcal{C}^{op}}_{1\ \ast,S}
\rightarrow {\bf M}^{\mathcal{C}^{op}}_{2\ \ast,S}$.
\end{proposition}
\begin{proof}
It is clear that the first pair is a Quillen pair.
For the second  it suffices \cite[8.5.4(3)]{Hi} 
to show that $G$ preserves fibrations between fibrant 
objects. But this is clear since $G$ preserves weak 
equivalences between fibrant objects.
The rest is a consequence of general facts about
left Bousfield localizations \cite[Chapter 3]{Hi}.
\end{proof}
Let us illustrate the last part of the previous proposition.
If $\mathcal{C}$ is an elegant Reedy category
\cite[Definition 3.5]{BR} and the cofibrations of
${\bf M}_{1}$ are the monomorphisms, then
$F':{\bf M}^{\mathcal{C}^{op}}_{1\ \ast}
\rightarrow {\bf M}^{\mathcal{C}^{op}}_{2\ \ast}$
preserves weak equivalences.

\subsection{The projective model structure}

\begin{proposition}
\cite[11.4.2]{Si} Let {\bf M} be a cofibrantly generated
model category. Let $\mathcal{C}$ be a small Reedy category
and $S$ a set. The category ${\bf M}^{\mathcal{C}^{op}}_{\ast}$
admits a cofibrantly generated model category
structure obtained by transfer from the projective
model structure on ${\bf M}^{\mathcal{C}^{op}}$
via the adjunction $(r,K)$ $(1.2(a))$. We denote 
this model structure by ${\bf M}^{\mathcal{C}^{op}}_{\ast,p}$.
${\bf M}^{\mathcal{C}^{op}}_{\ast,p}$ is left proper
if {\bf M} is. 
\end{proposition}
It follows from the proof of \cite[11.4.2]{Si} that a cofibrant
object of ${\bf M}^{\mathcal{C}^{op}}_{\ast,p}$ is
objectwise cofibrant. We also observe that the identity pair
$Id:{\bf M}^{\mathcal{C}^{op}}_{\ast,p}\rightleftarrows 
{\bf M}^{\mathcal{C}^{op}}_{\ast}:Id$
is a Quillen pair.

For the next result, recall from 2.1 the set $\mathfrak{S}$.
\begin{theorem}
{\rm (compare with \cite[12.1.1]{Si})} Let {\bf M}
be a left proper, combinatorial model category. Let $\mathcal{C}$ 
be a small Reedy category and $p:\mathcal{C}\rightarrow \Delta$
a degree preserving morphism which is a fibration. The category 
${\bf M}^{\mathcal{C}^{op}}_{\ast}$ admits a left proper, 
combinatorial model category structure with 
the cofibrations of ${\bf M}^{\mathcal{C}^{op}}_{\ast,p}$ 
as cofibrations and the $\mathfrak{S}$-local equivalences 
as weak equivalences. We denote this model structure by 
${\bf M}^{\mathcal{C}^{op}}_{\ast,p,S}$, 
where ``S'' stands for Segal. If {\bf X} is an object of 
${\bf M}^{\mathcal{C}^{op}}_{\ast}$, then {\bf X} is fibrant in 
${\bf M}^{\mathcal{C}^{op}}_{\ast,p,S}$ 
if and only if {\bf X} is fibrant in ${\bf M}^{\mathcal{C}^{op}}_{\ast,p}$ 
and for every object $c$ of $\mathcal{C}$ with $deg(c)\geq 1$, 
the map $${\bf X}(c)\rightarrow {\bf X}((\alpha^{0})^{\ast}c)
\times...\times {\bf X}((\alpha^{deg(c)-1})^{\ast}c)$$
is a weak equivalence of {\bf M}.
\end{theorem}
\begin{proof}
The proof is the same as for theorem 2.3, using proposition 2.11
instead of proposition 2.1.
\end{proof}
\begin{theorem}
{\rm (compare with \cite[12.3.2]{Si})}
The weak equivalences of ${\bf M}^{\mathcal{C}^{op}}_{\ast,p,S}$
and ${\bf M}^{\mathcal{C}^{op}}_{\ast,S}$ are the same.
\end{theorem}
\begin{proof}
Apply lemma 7.2 and the previous considerations.
\end{proof}

\section{Segal {\bf M}-categories and {\bf M}-categories 
(with fixed set of objects)}

Let {\bf M} be a cocomplete cartesian closed category.
Let $S$ be a set. We denote by {\bf M}\text{-}{\bf Cat}$(S)$ 
the category of small {\bf M}-categories with fixed set of objects $S$. 
Recall \cite{Si},\cite{Lu} that there is a functor 
$N:{\bf M}\text{-}{\bf Cat}(S) \rightarrow {\bf M}^{\Delta^{op}S}_{\ast}$ 
constructed as $$N\mathcal{A}(([n],s_{0},...,s_{n}))=
\begin{cases}
\ast, & \text{if }  n=0\\
\underline{\mathcal{A}}(s_{0},...,s_{n}), & \text{otherwise}
\end{cases} $$ 
where $\underline{\mathcal{A}}(s_{0},...,s_{n})=
\mathcal{A}(s_{0},s_{1})\times...\times \mathcal{A}(s_{n-1},s_{n})$.
For example, if $A$ is a monoid in {\bf M}, then $NA$ is the simplicial 
bar construction of the (trivially) augmented monoid $A$. 
$N$ is full and faithful. ${\bf X}\in {\bf M}^{\Delta^{op}S}_{\ast}$
is in the essential image of $N$ if and only if 
for every object $([n],s_{0},...,s_{n})$ of $\Delta S$, the canonical map
$${\bf X}(([n],s_{0},...,s_{n}))\rightarrow {\bf X}(([1],s_{0},s_{1}))
\times...\times {\bf X}(([1],s_{n-1},s_{n}))$$
is an isomorphism. $N$ has a left adjoint $L$ constructed
explicitly in \cite[2.2]{Lu}: to every pair $x,y$ of elements
of $S$ a certain category $\mathcal{J}_{x,y}(S)$ is associated,
and $L{\bf X}(x,y)$ is the colimit of a certain functor
$H^{{\bf X}}_{x,y}:\mathcal{J}_{x,y}(S)\rightarrow {\bf M}$
associated to {\bf X}.

For the next result, we regard {\bf M}\text{-}{\bf Cat}$(S)$
as having the standard \cite{SS2} model structure.
\begin{theorem}
{\rm (J. Lurie)} Let {\bf M} be a left proper, combinatorial 
cartesian model category with cofibrant unit, having a set of 
generating cofibrations with cofibrant domains and
satisfying the monoid axiom of \cite{SS1}.
Let $S$ be a set. Then the adjoint pair
$$L:{\bf M}^{\Delta^{op}S}_{\ast,p,S}
\rightleftarrows {\bf M}\text{-}{\bf Cat}(S):N$$ is a Quillen 
equivalence.
\end{theorem}
The above theorem was proved by J. Bergner \cite{Be1} 
in the case {\bf M}={\bf S}, using algebraic theories.
It was also proved in \cite[Theorem 2.2.16]{Lu}, 
not quite in this form, using a different method and 
under the assumption that all objects of {\bf M} are cofibrant. 
However, a close analysis of the proof of \cite[Theorem 2.2.16]{Lu}
reveals that this assumption is superfluous. To make
things clear we give below Lurie's proof, stripped to the
essentials and with some changes. At the heart of it
is the following technical result.
\begin{proposition}
{\rm \cite[Lemma 2.2.15]{Lu}} 
Let {\bf M} be a left proper, combinatorial, cartesian simplicial model category.
Let $S$ be a set. Let ${\bf X}\in {\bf M}^{\Delta^{op}S}_{\ast}$
be cofibrant in ${\bf M}^{\Delta^{op}S}_{\ast, p}$ and
such that for every object $([n],s_{0},...,s_{n})$ of $\Delta S$, 
the canonical map $${\bf X}(([n],s_{0},...,s_{n}))\rightarrow 
{\bf X}(([1],s_{0},s_{1}))\times...\times {\bf X}(([1],s_{n-1},s_{n}))$$
is a weak equivalence. Then for all pairs $x,y$ of 
elements of $S$, the canonical map ${\bf X}(([1],x,y))
\rightarrow NL{\bf X}(([1],x,y))$ is a weak equivalence of {\bf M}.
\end{proposition}
\begin{proof}
We use the notations of the proof of \emph{loc. cit.}.
Thus, we have a commutative diagram
\[
   \xymatrix{
{\bf X}(([1],x,y)) \ar[r]  & NL{\bf X}(([1],x,y))\\
colim_{\mathcal{J}_{x,y}'(S)}H \ar[r] \ar[u]^{\cong} & 
colim_{\mathcal{J}_{x,y}(S)}H^{{\bf X}}_{x,y} \ar[u]_{\cong}\\
hocolim_{\mathcal{J}_{x,y}'(S)}H \ar[r] \ar[u] & 
hocolim_{\mathcal{J}_{x,y}(S)}H^{{\bf X}}_{x,y} \ar[u]\\
}
  \]
The right vertical map is a weak equivalence by sublemma 3.3
since $H^{{\bf X}}_{x,y}$ is cofibrant \cite[Proposition 2.2.6]{Lu}.
The left vertical map is a weak equivalence since 
$\mathcal{J}_{x,y}'(S)$ has terminal object. We prove
that the bottom horizontal map is a weak equivalence.
Let $i$ be the inclusion $\mathcal{J}_{x,y}'(S)\subset
\mathcal{J}_{x,y}(S)$. There are a functor $R:
\mathcal{J}_{x,y}(S)\rightarrow \mathcal{J}_{x,y}'(S)$
and a natural transformation $\alpha:iR\Rightarrow Id$.
>From the second assumption on {\bf X} it follows that
for every object $\sigma$ of $\mathcal{J}_{x,y}(S)$,
the map $H^{{\bf X}}_{x,y}(\alpha_{\sigma})$ 
is a weak equivalence. We are then in the situation 
of sublemma 3.4 with $G=i$, $F=R$ and $Ri=Id$.
\end{proof}
For the next two (standard) results, $hocolim_{I}{\bf X}$
stands for the homotopy colimit of {\bf X}, as in \cite[18.1.2]{Hi}.
\begin{sublem}
Let {\bf M} be a cofibrantly generated simplicial model category
and $I$ a small category. Then for every
cofibrant object ${\bf X}\in {\bf M}^{I}$, the natural map
$$hocolim_{I}{\bf X}\rightarrow colim_{I}{\bf X}$$
is a weak equivalence.
\end{sublem}
\begin{sublem}
Let {\bf M} be a simplicial model category $I$ and $J$ 
two small categories. Suppose that are functors 
$G:I\rightarrow J$, and $F:J\rightarrow I$ together 
with natural transformations $\alpha:GF\Rightarrow Id_{J}$
and $\beta:FG\Rightarrow Id_{I}$. Let ${\bf X}:J\rightarrow {\bf M}$
take cofibrant values and be such that

$(a)$ ${\bf X}(\alpha_{j}):{\bf X}GF(j)\rightarrow {\bf X}(j)$ 
is a weak equivalence for all $j\in J$, and

$(b)$ ${\bf X}(G(\beta_{i})):{\bf X}GFG(i)\rightarrow {\bf X}G(i)$ 
is a weak equivalence for all $i\in I$.

Then the map $$hocolim_{I}{\bf X}G\rightarrow 
hocolim_{J}{\bf X}$$ is a weak equivalence.
\end{sublem}
We begin now the proof of theorem 3.1. It is clear
that $(L,N)$ is a Quillen pair and that 
$N$ preserves and reflects weak equivalences 
between fibrant objects. We prove that the total left 
derived functor of $L$ is full and faithful. This amounts to showing
that for every ${\bf X}\in {\bf M}^{\Delta^{op}S}_{\ast}$
which is cofibrant-fibrant in 
${\bf M}^{\Delta^{op}S}_{\ast, p,S}$, and for some
fibrant approximation $L{\bf X}\rightarrow \hat{F}L{\bf X}$
to $L${\bf X}, the map ${\bf X}\rightarrow NL{\bf X}\rightarrow 
N\hat{F}L{\bf X}$ is a weak equivalence in 
${\bf M}^{\Delta^{op}S}_{\ast, p,S}$.
We factor the map $L{\bf X}\rightarrow \ast$ as a trivial cofibration
$L{\bf X}\rightarrow \hat{F}L{\bf X}$ followed by a fibration
$\hat{F}L{\bf X}\rightarrow \ast$. We shall prove that  
${\bf X}\rightarrow NL{\bf X}\rightarrow 
N\hat{F}L{\bf X}$ is a weak equivalence in 
${\bf M}^{\Delta^{op}S}_{\ast, p}$.
Since $L{\bf X}$ is cofibrant, the map
$NL{\bf X}\rightarrow N\hat{F}L{\bf X}$ is a 
weak equivalence in ${\bf M}^{\Delta^{op}S}_{\ast, p}$.
Since {\bf X} is fibrant, to prove that 
${\bf X}\rightarrow NL{\bf X}$ is a weak equivalence
it suffices to prove that for every pair $x,y$ of elements of $S$, 
the map ${\bf X}(([1],x,y))\rightarrow L{\bf X}(x,y)$ is a weak 
equivalence.

Let $L_{\mathfrak{S}}{\bf M}^{\Delta^{op}}$ be the model category
considered in the proof of proposition 7.1. Since $cst:{\bf M}\rightarrow 
L_{\mathfrak{S}}{\bf M}^{\Delta^{op}}$ reflects weak 
equivalences between cofibrant objects, it suffices to show that 
$cst{\bf X}(([1],x,y))\rightarrow cstL{\bf X}(x,y)$ is a weak
equivalence. 

Consider the diagram
\[
   \xymatrix{
{\bf X}\in{\bf M}^{\Delta^{op}S}_{\ast} \ar @<2pt> [r]^{L} 
\ar @<2pt> [d]^{cst'} & {\bf M}\text{-}{\bf Cat}(S) 
\ar @<2pt> [l]^{N} \ar @<2pt> [d]^{cst}\\
(L_{\mathfrak{S}}{\bf M}^{\Delta^{op}})^{\Delta^{op}S}_{\ast}  
\ar @<2pt> [r]^{L} \ar @<2pt> [u]^{ev_{0}} &
L_{\mathfrak{S}}{\bf M}^{\Delta^{op}}\text{-}{\bf Cat}(S) 
\ar @<2pt> [l]^{N} \ar @<2pt> [u]^{ev_{0}}\\
  }
  \]
One has $ev_{0}N=Nev_{0}$, so $Lcst'\cong cstL$. 
We will prove that the object ${\bf Z}=cst'{\bf X}$ satisfies 
the assumptions of proposition 3.2. For this, it suffices 
to prove that for every object $([n],s_{0},...,s_{n})$ of $\Delta S$, 
the canonical map $${\bf Z}(([n],s_{0},...,s_{n}))\rightarrow 
{\bf Z}(([1],s_{0},s_{1}))\times...\times {\bf Z}(([1],s_{n-1},s_{n}))$$
is a weak equivalence in ${\bf M}^{\Delta^{op}}$. Under
the isomorphism $({\bf M}^{\Delta^{op}})^{\Delta^{op}S}_{\ast}
\cong ({\bf M}^{\Delta^{op}S}_{\ast})^{\Delta^{op}}$
(1.2$(d)$) {\bf Z} corresponds to $cst{\bf X}$, and then 
the fact that {\bf X} is fibrant implies that the required map
is a weak equivalence. Thus, by proposition 3.2 the map
${\bf Z}(([1],x,y))\rightarrow L{\bf Z}(x,y)$ is a weak equivalence.
But $L{\bf Z}(x,y)\cong cstL{\bf X}(x,y)$ and 
${\bf Z}(([1],x,y))\cong cst{\bf X}(([1],x,y))$. The proof
of theorem 3.1 is complete.

\section{The category of pre-{\bf M}-categories}

Recall from example 2.4 the category $\Delta S$.
Let {\bf M} be a category. A function $u:S\rightarrow T$ induces
a functor $u^{\ast}:{\bf M}^{\Delta^{op}T}\rightarrow
{\bf M}^{\Delta^{op}S}$, which has a left adjoint $u_{!}$
 provided that {\bf M} is cocomplete, and a right adjoint $u_{\ast}$
provided that {\bf M} is complete. 
\begin{definition}
We define a category ${\bf M}^{\Delta^{op}Set}$ as follows. 
The objects of ${\bf M}^{\Delta^{op}Set}$ are pairs 
$(S,{\bf X})$, where $S$ is a set and ${\bf X}\in {\bf M}^{\Delta^{op}S}$. 
An arrow $(S,{\bf X})\rightarrow (T,{\bf Y})$ is a pair $(u,f)$, 
where $u:S\rightarrow T$ is a function and 
$f:{\bf X}\Rightarrow u^{\ast}{\bf Y}$ is a natural tranformation.
\end{definition}
Suppose that {\bf M} is suitably complete and cocomplete.
The terminal object of ${\bf M}^{\Delta^{op}Set}$ is $(\ast,\ast)$.
The functor $Ob:{\bf M}^{\Delta^{op}Set}\rightarrow Set$ defined as
$Ob((S,{\bf X}))=S$ has a left adjoint $D$ given by $DS=(S,\emptyset)$, 
and a right adjoint $\iota$ given by $\iota S=(S,\ast)$. 
$D$ and $\iota$ are full and faithful. The functor $Ob$ is a 
cloven Grothendieck bifibration. An arrow 
$(u,f):(S,{\bf X})\rightarrow (T,{\bf Y})$ is cartesian if 
and only if $f$ is an isomorphism. The fibre category of 
$Ob$ over a set $S$ is ${\bf M}^{\Delta^{op}S}$.

Suppose that {\bf M} is a closed category. 
Write $Y^{X}$ for the internal hom 
of two objects $X$, $Y$ of {\bf M}. Then for every
set $S$, ${\bf M}^{\Delta^{op}S}$ is tensored and 
cotensored over {\bf M} (1.2$(c)$). For every function 
$u:S\rightarrow T$, ${\bf X}\in {\bf M}^{\Delta^{op}T}$ and
$A\in {\bf M}$ we have the formula 
$u^{\ast}({\bf X}^{A})=(u^{\ast}{\bf X})^{A}$, 
hence ${\bf M}^{\Delta^{op}Set}$ is tensored and cotensored
over {\bf M}, with tensor $A\bigstar (S,{\bf X})=(S,A\bigstar{\bf X})$ 
and cotensor $(S,{\bf X})^{A}=(S,{\bf X}^{A})$.

\subsection{The category of pre-{\bf M}-categories}
Let $S$ be a set. Recall from section 1 the inclusion 
$\sigma_{0}:S\subset \Delta S$.
\begin{definition}
Let {\bf M} be a category with terminal object.
The category ${\bf M}^{\Delta^{op}Set}_{\ast}$ is 
the full subcategory of ${\bf M}^{\Delta^{op}Set}$ on objects 
$(S,{\bf X})$ such that $\sigma_{0}^{\ast}{\bf X}=\ast$.
The objects of ${\bf M}^{\Delta^{op}Set}_{\ast}$ are referred to as
{\bf pre}-{\bf M}-{\bf categories}, and the arrows of 
${\bf M}^{\Delta^{op}Set}_{\ast}$ as 
{\bf pre}-{\bf M}-{\bf functors}. If $(S,{\bf X})$ is a pre-{\bf M}-category,
then $S$ is its {\bf set of objects}.
\end{definition}
The category ${\bf M}^{\Delta^{op}Set}_{\ast}$
is denoted by $\mathcal{PC}(\mathcal{M})$ in \cite{Si}.
We let $K$ be the inclusion functor 
${\bf M}^{\Delta^{op}Set}_{\ast}
\subset {\bf M}^{\Delta^{op}Set}$.
$K$ has a left adjoint $r$ calculated 
(1.2$(a)$) as $r((S,{\bf X}))=(S,r{\bf X})$.
We denote by $Ob$ the composite $ObK$, and by 
$D$ the composite $rD$. One has
$$DS(([n],s_{0},...,s_{n}))=
 \begin{cases}
\ast, & \text{if }  s_{0}=...=s_{n}\\
 \emptyset, & \text{otherwise}
\end{cases} $$ 
Note that $\iota:Set\rightarrow {\bf M}^{\Delta^{op}Set}$
takes values in ${\bf M}^{\Delta^{op}Set}_{\ast}$, $D$ and 
$\iota$ are full and faithful, $D$ is left adjoint to $Ob$ and $\iota$ 
is right adjoint to $Ob$. The functor $Ob$ is a cloven Grothendieck 
bifibration. An arrow $(u,f):(S,{\bf X})\rightarrow (T,{\bf Y})$ is 
cartesian if and only if $f$ is an isomorphism \cite[10.3]{Si}. 
The fibre category of $Ob$ over a set $S$ is
${\bf M}^{\Delta^{op}S}_{\ast}$. $K$ becomes a 
cartesian functor. To give a map 
$(\ast,\ast)\rightarrow (S,{\bf X})$ in
${\bf M}^{\Delta^{op}Set}_{\ast}$ is to 
give an object of {\bf X}.

We shall compute $Set^{\Delta^{op}Set}_{\ast}$. Recall
that {\bf S} denotes the category of simplicial sets. Recall also
that for every small category $\mathcal{C}$ and every 
$X\in Set^{\mathcal{C}^{op}}$ there is an equivalence of categories 
$$(Set^{\mathcal{C}^{op}}\downarrow X)\overset{\simeq}
\rightarrow Set^{(y\downarrow X)^{op}}$$ 
In particular, for every set $U$ there is an equivalence
of categories $$({\bf S}\downarrow N\iota U)\overset{\simeq}
\rightarrow Set^{\Delta^{op}U}$$ 
Under the above equivalence ${\bf S}_{U}$
corresponds to $Set^{\Delta^{op}U}_{\ast}$. It
follows that {\bf S} is equivalent to 
$Set^{\Delta^{op}Set}_{\ast}$. 

More generally, we compute 
$(Set^{\mathcal{C}})^{\Delta^{op}Set}_{\ast}$.
Let $U$ be a set. From the previous calculation and 1.2$(d)$
it follows that $(Set^{\mathcal{C}})^{\Delta^{op}U}_{\ast}$
is equivalent to the full subcategory of 
$Set^{\Delta^{op}\times \mathcal{C}}$ consisting of those objects
$X:\Delta^{op}\times \mathcal{C}\rightarrow Set$ such that 
$X([0],c)=U$ for all $c\in \mathcal{C}$. If $\mathcal{C}$ 
is connected, $(Set^{\mathcal{C}})^{\Delta^{op}Set}_{\ast}$
is equivalent with the full subcategory of 
$Set^{\Delta^{op}\times \mathcal{C}}$ 
consisting of those objects which take every map in 
$\{[0]\}\times \mathcal{C}$ to an isomorphism. 
For example, ${\bf S}^{\Delta^{op}Set}_{\ast}$
is the full subcategory of ${\bf S}^{\Delta^{op}}$ on 
those bisimplicial sets $X$ with $X_{0}$ a constant/discrete 
simplicial set. For a general $\mathcal{C}$,
let $\mathcal{D}$ be the pushout
\[
   \xymatrix{
\{[0]\}\times \mathcal{C} \ar[r] \ar[d] & 
\Delta^{op} \times \mathcal{C} \ar[d]\\
\ast \ar[r] & \mathcal{D}\\
}
  \]
Then $(Set^{\mathcal{C}})^{\Delta^{op}Set}_{\ast}$
is equivalent to $Set^{\mathcal{D}}$.  

Suppose that {\bf M} is a closed category. Then
${\bf M}^{\Delta^{op}Set}_{\ast}$ is cotensored over {\bf M}, 
with cotensor $(S,{\bf X})^{A}=(S,(K{\bf X})^{A})$.
Since $K((S,{\bf X})^{A})=(K(S,{\bf X}))^{A}$, it follows that 
${\bf M}^{\Delta^{op}Set}_{\ast}$ is tensored over {\bf M}, 
with tensor $A\bigstar (S,{\bf X})=r(A\bigstar K(S,{\bf X}))$.

Let $(u,f):(S,{\bf X})\rightarrow (T,{\bf Y})$ be a map in 
${\bf M}^{\Delta^{op}Set}_{\ast}$. Then $(u,f)$
decomposes (like any map of a bifibration)
as $(S,{\bf X})\rightarrow (S,u^{\ast}{\bf Y})
\rightarrow (T,{\bf Y})$ and as $(S,{\bf X})\rightarrow (T,u_{!}{\bf X})
\rightarrow (T,{\bf Y})$. When $u$ is a monomorphism, 
$u_{!}$ has a convenient description \cite[10.3]{Si}:
$$u_{!}{\bf X}(([n],t_{0},...,t_{n}))=
  \begin{cases}
{\bf X}(([n],s_{0},...,s_{n})), & \text{if }  
t_{i}=u(s_{i}), 0\leq i\leq n\\
\ast, & \text{if } t_{0}=...=t_{n}\in T-S\\
 \emptyset, & \text{otherwise}
\end{cases} $$ 
\begin{lem}
Let $(u,f):(S,{\bf X})\rightarrow (T,{\bf Y})$ be a map in 
${\bf M}^{\Delta^{op}Set}_{\ast}$. The following 
are equivalent:

$(a)$ $(u,f)$ is a monomorphism;

$(b)$  $u$ is a monomorphism and 
$f:{\bf X}\rightarrow u^{\ast}{\bf Y}$ is a
monomorphism in 
${\bf M}^{\Delta^{op}S}_{\ast}$;

$(c)$ $u$ is a monomorphism and 
$u_{!}{\bf X}\rightarrow {\bf Y}$ is a 
monomorphism in 
${\bf M}^{\Delta^{op}T}_{\ast}$.
\end{lem}
\begin{proof}
$(a)\Leftrightarrow (b)$ is a standard argument.
$(b)\Leftrightarrow (c)$ by the description 
of $u_{!}{\bf X}$.
\end{proof}
\subsection{Relation with enriched categories}
Let {\bf M} be a cocomplete cartesian closed category. 
We denote by {\bf M}\text{-}{\bf Cat} the category of small 
categories enriched over {\bf M}. We recall \cite{Si},\cite{Lu} that
there is a functor $N:{\bf M}\text{-}{\bf Cat}\rightarrow 
{\bf M}^{\Delta^{op}Set}_{\ast}$ constructed as follows. 
If $\mathcal{A}$ is an {\bf M}-category, 
$N\mathcal{A}=(Ob\mathcal{A},N\mathcal{A})$,
where $$N\mathcal{A}(([n],s_{0},...,s_{n}))=
\begin{cases}
\ast, & \text{if }  n=0\\
\underline{\mathcal{A}}(s_{0},...,s_{n}), & \text{otherwise}
\end{cases} $$ 
and $\underline{\mathcal{A}}(s_{0},...,s_{n})=
\mathcal{A}(s_{0},s_{1})\times...\times \mathcal{A}(s_{n-1},s_{n})$.
$N$ is full and faithful and cartesian. A pre-{\bf M}-category {\bf X}
is in the essential image of $N$ if and only if 
for every object $([n],s_{0},...,s_{n})$ of $\Delta S$, 
the canonical map $${\bf X}(([n],s_{0},...,s_{n}))\rightarrow 
{\bf X}(([1],s_{0},s_{1}))\times...
\times {\bf X}(([1],s_{n-1},s_{n}))$$
is an isomorphism.

\section{Fibred Reedy model structures on pre-{\bf M}-categories}

For the next result, recall from example 2.4 the category $\Delta S$
and from proposition 2.1 the model category ${\bf M}^{\Delta^{op}S}_{\ast}$.
\begin{theorem}
Let {\bf M} be a model category.
The category ${\bf M}^{\Delta^{op}Set}_{\ast}$
admits a model category structure, denoted by
$f{\bf M}^{\Delta^{op}Set}_{\ast}$, in which a map
 $(u,f):(S,{\bf X})\rightarrow (T,{\bf Y})$  is a 

$\bullet$ weak equivalence if $u$ is bijective and 
$f:{\bf X}\rightarrow u^{\ast}{\bf Y}$ is a weak equivalence
in ${\bf M}^{\Delta^{op}S}_{\ast}$,

$\bullet$ cofibration if $u_{!}{\bf X}\rightarrow {\bf Y}$ 
is a cofibration in ${\bf M}^{\Delta^{op}T}_{\ast}$,

$\bullet$ fibration if $f:{\bf X}\rightarrow u^{\ast}{\bf Y}$
is a fibration in ${\bf M}^{\Delta^{op}S}_{\ast}$.
\end{theorem}
\begin{proof}
The functor $Ob:{\bf M}^{\Delta^{op}Set}_{\ast}\rightarrow Set$
is a bifibration (section 4.1). Let $Set$ have the 
minimal model structure, in which the weak
equivalences are the isomorphisms and all 
maps are cofibrations as well as fibrations.
Using corollary 2.8 one can check that the conditions 
of \cite[2.3, Theorem]{St} are satisfied.
\end{proof}
\begin{definition}
Let {\bf M} be a left proper, combinatorial model category. 
A map $(u,f):(S,{\bf X})\rightarrow (T,{\bf Y})$ 
of ${\bf M}^{\Delta^{op}Set}_{\ast}$ is a 

$(a)$  {\bf fibred weak equivalence} if $u$ is bijective and 
$f:{\bf X}\rightarrow u^{\ast}{\bf Y}$ is a weak equivalence
in ${\bf M}^{\Delta^{op}S}_{\ast,S}$;

$(b)$ {\bf fibred cofibration} if $u_{!}{\bf X}\rightarrow {\bf Y}$ 
is a cofibration in ${\bf M}^{\Delta^{op}T}_{\ast,S}$;

$(c)$ {\bf fibred fibration} if $f:{\bf X}\rightarrow u^{\ast}{\bf Y}$
is a fibration in ${\bf M}^{\Delta^{op}S}_{\ast,S}$.
\end{definition}
A map which is both a fibred weak equivalence and fibred cofibration is
called isotrivial cofibration in \cite[14.3]{Si}.
\begin{theorem}
Let {\bf M} be a left proper, combinatorial model category. The classes of
fibred weak equivalences, fibred cofibrations and fibred fibrations
form a model category structure on ${\bf M}^{\Delta^{op}Set}_{\ast}$.
We denote this model structure by $f{\bf M}^{\Delta^{op}Set}_{\ast,S}$.
$f{\bf M}^{\Delta^{op}Set}_{\ast,S}$ is a left Bousfield 
localization of $f{\bf M}^{\Delta^{op}Set}_{\ast}$.
\end{theorem}
\begin{proof}
The functor $Ob:{\bf M}^{\Delta^{op}Set}_{\ast}\rightarrow Set$
is a bifibration (section 4.1). Let $Set$ have the
minimal model structure, in which the weak
equivalences are the isomorphisms and all 
maps are cofibrations as well as fibrations. For every
set $S$, we let the category ${\bf M}^{\Delta^{op}S}_{\ast}$
have the model structure ${\bf M}^{\Delta^{op}S}_{\ast,S}$
of theorem 2.3. Using corollary 2.9 one can check that the conditions 
of \cite[2.3, Theorem]{St} are satisfied. The rest is clear.
\end{proof}
\subsection{A weak factorization system}
\begin{definition}
Let {\bf M} be a left proper, combinatorial model category.
We say that a map $(u,f):(S,{\bf X})\rightarrow (T,{\bf Y})$ 
of ${\bf M}^{\Delta^{op}Set}_{\ast}$ is a 

$(a)$ {\bf Reedy cofibration} if $u$ is injective and 
$u_{!}{\bf X}\rightarrow {\bf Y}$ is a cofibration
in ${\bf M}^{\Delta^{op}T}_{\ast,S}$;

$(b)$ {\bf Reedy trivial fibration} if $u$ is surjective and 
$f:{\bf X}\rightarrow u^{\ast}{\bf Y}$ is a trivial fibration
in ${\bf M}^{\Delta^{op}S}_{\ast,S}$.
\end{definition}
\begin{proposition}
Let {\bf M} be a left proper, combinatorial model category.
The pair {\rm (Reedy cofibrations, Reedy trivial fibrations)}
is a weak factorization system on ${\bf M}^{\Delta^{op}Set}_{\ast}$.
\end{proposition}
\begin{proof}
The functor $Ob:{\bf M}^{\Delta^{op}Set}_{\ast}\rightarrow Set$
is a bifibration. Let $Set$ have the weak factorization system
(monomorphisms, epimorphisms) and, for every set $S$, 
${\bf M}^{\Delta^{op}S}_{\ast,S}$ have the weak factorization system
(cofibrations, trivial fibrations). Then corollary 2.9 and \cite[2.2]{St}
imply the proposition.
\end{proof}
For every simplicial set $X$, $(y\downarrow X)$ (example 2.4) 
is an EZ-Reedy category \cite[Definition 4.1]{BR}.
The next result is then a consequence of lemma 4.3.
\begin{corollary} {\rm (compare with \cite[13.7.2]{Si})}
Let {\bf M} be a combinatorial model category such 
that the cofibrations of {\bf M} are the monomorphisms.
Then a map of ${\bf M}^{\Delta^{op}Set}_{\ast}$
is a Reedy cofibration if and only if it is a monomorphism.
\end{corollary}
We recall \cite[Definition 3.16.2]{Ba} that a morphism $f:\mathcal{C}
\rightarrow \mathcal{D}$ of Reedy categories is a {\bf left fibration}
if for every model category {\bf M}, the adjoint pair
$f^{\ast}:{\bf M}^{\mathcal{D}^{op}}\rightleftarrows 
{\bf M}^{\mathcal{C}^{op}}:f_{\ast}$ is a Quillen pair.
\begin{lem}
\cite[Proof of Theorem 3.51]{Ba} Let 
$f:\mathcal{C}\rightarrow \mathcal{D}$ be a morphism
of Reedy categories. Suppose that $\mathcal{C}$ has fibrant 
constants and every arrow of the inverse subcategory 
$\overleftarrow{\mathcal{D}}$ is an epimorphism. 
Then $f$ is a left fibration.
\end{lem}
\begin{proof}
It suffices to prove that $\partial(\alpha\downarrow 
\overleftarrow{(\overleftarrow{f}\downarrow d)})$
is empty or connected. Let $\alpha=(c,f(c)\rightarrow d)$.
Since $\partial(c\downarrow \overleftarrow{\mathcal{C}})$
is empty or connected, it all reduces to proving,
as in \cite[Proof of Theorem 3.51]{Ba},  
that the lower triangle in a diagram
\[
   \xymatrix{
& f(c) \ar[dr] \ar[dl] \ar[dd]\\
f(x) \ar @{..>} [rr] \ar[dr] & & f(y) \ar[dl]\\
& d\\
}
  \]
commutes. This is where the assumption on
$\overleftarrow{\mathcal{D}}$ comes in.
\end{proof}
\begin{corollary}
For any map $f:X\rightarrow Y$ is of simplicial sets and any 
model category {\bf M}, the adjoint pair $$(y\downarrow f)^{\ast}:
{\bf M}^{(y\downarrow Y)^{op}}\rightleftarrows
{\bf M}^{(y\downarrow X)^{op}}:(y\downarrow f)_{\ast}$$
is a Quillen pair.
\end{corollary}
\begin{lem}
Let {\bf M} be a left proper, combinatorial model category. 
If $(u,f):(S,{\bf X})\rightarrow (T,{\bf Y})$ is a Reedy 
cofibration, then $f:{\bf X}\rightarrow u^{\ast}{\bf Y}$ 
is a cofibration in ${\bf M}^{\Delta^{op}S}_{\ast,S}$.
\end{lem}
\begin{proof}
Let $f_{u}$ be the map $u_{!}{\bf X}\rightarrow {\bf Y}$.
By definition, $f_{u}$ is a cofibration in 
${\bf M}^{\Delta^{op}T}_{\ast}$, hence by corollary 5.8
suitably applied (see 1.2$(f)$), $u^{\ast}f_{u}$ is a 
cofibration in ${\bf M}^{\Delta^{op}S}_{\ast}$.
But $u^{\ast}f_{u}\cong f$ since $u$ is a monomorphism.
\end{proof}
\begin{lem}
Let {\bf M} be a left proper, combinatorial model category. Then fibred weak 
equivalences are stable under pushout along Reedy cofibrations.
\end{lem}
\begin{proof}
Since $f{\bf M}^{\Delta^{op}Set}_{\ast,S}$ is a left Bousfield 
localization of $f{\bf M}^{\Delta^{op}Set}_{\ast}$, it suffices to 
prove the lemma for the weak equivalences of 
$f{\bf M}^{\Delta^{op}Set}_{\ast}$. So let 
$(v,f):(R,{\bf X})\rightarrow (S,{\bf Y})$ be a weak equivalence
of $f{\bf M}^{\Delta^{op}Set}_{\ast}$ and 
$(u,g):(R,{\bf X})\rightarrow (T,{\bf Z})$ a Reedy cofibration.
Without loss of generality we may assume that $v$ is the
identity map of $S$. The pushout of $(1_{S},f)$ along $(u,g)$
is calculated as the pushout
\[
   \xymatrix{
u_{!}{\bf X} \ar[r]^{u_{!}f} \ar[d] & u_{!}{\bf Y} \ar[d]\\
{\bf Z} \ar[r] & {\bf P}\\
}
  \]
in ${\bf M}^{\Delta^{op}T}_{\ast}$. From the
explicit description of $u_{!}{\bf X}$ and $u_{!}{\bf Y}$
the map $u_{!}f$ is a weak equivalence in 
${\bf M}^{\Delta^{op}T}_{\ast}$, and the lemma
follows since ${\bf M}^{\Delta^{op}T}_{\ast}$
is left proper.
\end{proof}

\section{Fibred projective model structures on 
pre-{\bf M}-categories}

For the next result, recall from example 2.4 the 
category $\Delta S$ and from proposition 2.11 the 
model category ${\bf M}^{\Delta^{op}S}_{\ast,p}$.

\begin{theorem}
Let {\bf M} be a cofibrantly generated model category.
The category ${\bf M}^{\Delta^{op}Set}_{\ast}$
admits a model category structure, denoted by
$f{\bf M}^{\Delta^{op}Set}_{\ast,p}$, in which a map
 $(u,f):(S,{\bf X})\rightarrow (T,{\bf Y})$  is a 

$\bullet$ weak equivalence if $u$ is bijective and 
$f:{\bf X}\rightarrow u^{\ast}{\bf Y}$ is a weak equivalence
in ${\bf M}^{\Delta^{op}S}_{\ast,p}$,

$\bullet$ cofibration if $u_{!}{\bf X}\rightarrow {\bf Y}$ 
is a cofibration in ${\bf M}^{\Delta^{op}T}_{\ast,p}$,

$\bullet$ fibration if $f:{\bf X}\rightarrow u^{\ast}{\bf Y}$
is a fibration in ${\bf M}^{\Delta^{op}S}_{\ast,p}$.
\end{theorem}
\begin{proof}
The proof is the same as for theorem 5.1, using the fact 
that if $f:S\rightarrow T$ is a function, the adjoint pair 
$f_{!}':{\bf M}^{\Delta^{op}S}_{\ast,p}
\rightleftarrows {\bf M}^{\Delta^{op}T}_{\ast,p}:f^{\ast}$
is a Quillen pair.
\end{proof}
\begin{definition}
Let {\bf M} be a left proper, combinatorial model category. 
A map $(u,f):(S,{\bf X})\rightarrow (T,{\bf Y})$ 
of ${\bf M}^{\Delta^{op}Set}_{\ast}$ is a 

$(a)$  {\bf fibred projective weak equivalence} if $u$ is bijective and 
$f:{\bf X}\rightarrow u^{\ast}{\bf Y}$ is a weak equivalence
in ${\bf M}^{\Delta^{op}S}_{\ast,p,S}$;

$(b)$ {\bf fibred projective cofibration} if $u_{!}{\bf X}\rightarrow {\bf Y}$ 
is a cofibration in ${\bf M}^{\Delta^{op}T}_{\ast,p,S}$;

$(c)$ {\bf fibred projective fibration} if $f:{\bf X}\rightarrow u^{\ast}{\bf Y}$
is a fibration in ${\bf M}^{\Delta^{op}S}_{\ast,p,S}$.
\end{definition}
A map which is both a fibred projective weak equivalence and 
fibred projective cofibration is called isotrivial cofibration in \cite[14.3]{Si}.
\begin{theorem}
Let {\bf M} be a left proper, combinatorial model category. The classes of
fibred projective weak equivalences, fibred projective cofibrations and
fibred projective fibrations form a model category structure on 
${\bf M}^{\Delta^{op}Set}_{\ast}$. We denote this model 
structure by $f{\bf M}^{\Delta^{op}Set}_{\ast,p,S}$.
$f{\bf M}^{\Delta^{op}Set}_{\ast,p,S}$ is a left Bousfield 
localization of $f{\bf M}^{\Delta^{op}Set}_{\ast,p}$.
\end{theorem}
\begin{proof}
The proof is the same as for theorem 5.3, using the fact that
if $f:S\rightarrow T$ is a function, the adjoint pair 
$f_{!}':{\bf M}^{\Delta^{op}S}_{\ast,p,S}
\rightleftarrows {\bf M}^{\Delta^{op}T}_{\ast,p,S}:f^{\ast}$
is a Quillen pair.
\end{proof}
The next result was proved in the case when {\bf M} is the
category of simplicial sets by A. Joyal (unpublished).
It is a straightforward application of theorem 3.1.
\begin{theorem}
Let {\bf M} be a left proper, combinatorial cartesian
model category with cofibrant unit, having a set of 
generating cofibrations with cofibrant domains and
satisfying the monoid axiom of \cite{SS1}. Regard
{\bf M}\text{-}{\bf Cat} as having the fibred model structure
\cite[4.2]{St}. Then the adjoint pair 
$$L:{\bf M}^{\Delta^{op}Set}_{\ast,p,S}
\rightleftarrows {\bf M}\text{-}{\bf Cat}:N$$ is a Quillen 
equivalence.
\end{theorem}

\section{Appendix: Two facts about left Bousfield localization}

Let {\bf M} be a model category and let $C$ be a class of maps of {\bf M}.
We denote by $L_{C}{\bf M}$ the left Bousfield localization of {\bf M}
with respect to $C$ \cite[3.3.1(1)]{Hi}. Recall that {\bf S} denotes
the category of simplicial sets.
\begin{proposition}
A left proper, combinatorial monoidal model category 
with cofibrant unit and having a set of generating cofibrations 
with cofibrant domains is Quillen equivalent
to a monoidal simplicial model category via a strong monoidal
adjunction.
\end{proposition}
\begin{proof}
Let {\bf M} be as in the hypothesis, with unit $I$. We regard 
${\bf M}^{\Delta^{op}}$ as having the Reedy model structure.
By \cite[Theorem 3.51]{Ba} ${\bf M}^{\Delta^{op}}$ is a 
monoidal model category with cofibrant unit. 
Let $y:\Delta\rightarrow {\bf M}^{\Delta^{op}}$ 
be the functor $y([n])=\underset{\Delta(-,[n])}\sqcup I$ and  
$\mathfrak{S}$ be the set of maps $\{y([n])\rightarrow 
y([0])\}_{[n]\in\Delta}$ of ${\bf M}^{\Delta^{op}}$. 
The model structure $L_{\mathfrak{S}}{\bf M}^{\Delta^{op}}$ 
is a particular case of the one obtained by D. Dugger 
\cite[Theorem 5.7]{Du}, and it is also
the left Bousfield localization of ${\bf M}^{\Delta^{op}}$
with respect to $\mathfrak{S}$ enriched over {\bf M} 
\cite[Definition 4.42 and Theorem 4.46]{Ba}.
By \cite[Proposition 4.47]{Ba} $L_{\mathfrak{S}}{\bf M}^{\Delta^{op}}$
is a monoidal model category. Among other things, Dugger proves 
that $L_{\mathfrak{S}}{\bf M}^{\Delta^{op}}$ 
is a simplicial model category. It follows that 
$L_{\mathfrak{S}}{\bf M}^{\Delta^{op}}$ is a monoidal 
simplicial model category (also known as monoidal {\bf S}-model 
category). By \cite[Theorem 6.1]{Du} $$cst:{\bf M}
\rightleftarrows L_{\mathfrak{S}}{\bf M}^{\Delta^{op}}:ev_{0}$$
is a Quillen equivalence, where $cst$ is the constant
simplicial object functor and $ev_{0}$ is the evaluation at $[0]$.
\end{proof}
\begin{lem} \cite[Proof of Theorem 5.7(a)]{Du}
Let {\bf M}$^{(1)}$ and {\bf M}$^{(2)}$ be two model category
structures on a category {\bf M} such that
the identity pair $Id:{\bf M}^{(1)}\rightleftarrows {\bf M}^{(2)}:Id$
is a Quillen pair and the weak equivalences of {\bf M}$^{(1)}$
and {\bf M}$^{(2)}$ are the same. Let $C$ be a class of maps
of {\bf M}. Then the weak equivalences of the left Bousfield
localizations of {\bf M}$^{(1)}$ and {\bf M}$^{(2)}$ with
respect to $C$ are the same.
\end{lem}
\begin{proof}
Part of the hypothesis implies that every weak equivalence of $L_{C}{\bf M}^{(1)}$
is a weak equivalence of $L_{C}{\bf M}^{(2)}$. Conversely,
let $f:X\rightarrow Y$ be a weak equivalence of $L_{C}{\bf M}^{(2)}$.
Let $\bar{f}:\bar{X}\rightarrow \bar{Y}$ be a fibrant approximation
to $f$ in $L_{C}{\bf M}^{(1)}$ and $\tilde{f}:\tilde{X}\rightarrow \tilde{Y}$
a fibrant approximation to $\bar{f}$ in {\bf M}$^{(2)}$. It follows
that $\tilde{X}$ and $\tilde{Y}$ are $C$-local with respect to {\bf M}$^{(2)}$,
therefore $\tilde{f}$ is a weak equivalence of {\bf M}$^{(i)}, i\in \{1,2\}$.
But then $f$ is a weak equivalence of $L_{C}{\bf M}^{(1)}$.
\end{proof}

\end{document}